\documentclass[12pt]{amsart}

\author{Simeon Ball and Enrique Jimenez}\thanks{
The first author acknowledges the support of the project  MTM2017-82166-P of the Spanish {\em Ministerio de Econom\'ia y Competitividad.}}

 \usepackage{amsmath}
\usepackage{amssymb}
\usepackage{amsthm}
\usepackage{graphicx}
\usepackage{amscd}
\usepackage{amsthm,amsmath,amssymb,amsfonts,amscd}

\usepackage{enumerate}
\usepackage{epstopdf}
\usepackage{booktabs}
\usepackage{subcaption}
\usepackage{array}
\usepackage{multirow}
\usepackage{listings}

\newtheorem{theorem}{Theorem}
\newtheorem{lemma}[theorem]{Lemma}

 \newtheorem{conjecture}[theorem]{Conjecture}

\title{On sets defining few ordinary solids}

 \usepackage{amsmath}
\usepackage{amssymb}
\usepackage{amsthm}
\usepackage{graphicx}
\usepackage{amscd}

\begin{document}

\baselineskip=17pt

\date{20 October 2020}

\maketitle

%\keywords{ordinary planes, eight associated points theorem, sylvester-gallai, green-tao}

\begin{abstract}
Let $\mathcal{S}$ be a set of $n$ points in real four-dimensional space, no four coplanar and spanning the whole space. We prove that if the number of solids incident with exactly four points of $\mathcal{S}$ is less than $Kn^3$ for some $K=o(n^{\frac{1}{7}})$ then, for $n$ sufficiently large, all but at most $O(K)$ points of $\mathcal{S}$ are contained in the intersection of five linearly independent quadrics. Conversely, we prove that there are finite subgroups of size $n$ of an elliptic curve which span less than $\frac{1}{6}n^3$ solids containing exactly four points of $\mathcal{S}$.
\end{abstract}

\section{Introduction}

This work is based on the articles of Green and Tao \cite{GT2013} and the first author \cite{Ball2018}. Green and Tao proved that if $\mathcal{S}$ is a set of $n$ points in the real plane with the property that the number of lines incident with exactly two points of $\mathcal{S}$ is less than $Kn$,  where $K<c(\log \log n)^c$ for some constant $c$ and $n$ is sufficiently large, then all but at most $O(K)$ points of $\mathcal{S}$ are contained in a (possibly degenerate) cubic curve. The lines incident with exactly two points of $\mathcal{S}$ are called {\em ordinary lines}. In fact Green and Tao give a complete classification of the sets of $n$ points $\mathcal{S}$ spanning less than $cn(\log \log n)^c$ ordinary lines. Their paper inspired not only \cite{Ball2018} but also articles concerning ordinary circles \cite{LMMSSdZ2016} and ordinary conics \cite{BVZ2016} and \cite{CDFGLMSST2016}.

The first author proved in \cite{Ball2018} that if $\mathcal{S}$ is a set of $n$ points in real three-dimensional space with the property that no three points are collinear, $\mathcal{S}$ spans the whole space, and $\mathcal{S}$ spans at most $Kn^2$ planes incident with exactly three points of $\mathcal{S}$, where $K=o(n^{1/7})$ and $n$ is sufficiently large, then all but at most $O(K)$ points of $\mathcal{S}$ are contained in the intersection of two quadrics. The planes incident with exactly three points of $\mathcal{S}$ are called {\em ordinary planes}.

Suppose that $\mathcal{S}$ is a set of $n$ points in four-dimensional space containing no four co-planar points and spanning the whole space. We say that a hyperplane, a three-dimensional subspace, is an {\em ordinary solid}, if it is incident with exactly four points of $\mathcal{S}$. In this article we will prove that if $\mathcal{S}$ spans at most $Kn^3$ ordinary solids, where $K=o(n^{1/7})$ and $n$ is sufficiently large, then all but at most $O(K)$ points of $\mathcal{S}$ are contained in the intersection of five linearly independent quadrics. We will, for the most part, refer to an ordinary solid as an ordinary hyperplane, to ease comparison with results in other dimensions.

In the next section we shall provide examples of sets of $n$ points in four-dimensional real space which span less than $\frac{1}{6}n^3$ ordinary solids. This implies that there are examples attaining the bound for $K=\frac{1}{6}$.

Let ${\mathbb P}^d({\mathbb R})$ denote the $d$-dimensional projective space over ${\mathbb R}$. The notation $$(x_1:x_2:\ldots :x_{d+1})$$ denotes the projective point with vector representative $(x_1,x_2,\ldots,x_{d+1})$.

\section{Elliptic curves}

Let $\mathcal{E}$ be the elliptic curve in the plane defined as the zeros of the polynomial
$$
X_3^2X_1-X_2^3-aX_2X_1^2-bX_1^3.
$$
The points on this elliptic curve are
$$
\{(1:x : y ) \ | \ y^2=x^3+ax+b \} \cup \{ O \},
$$
where $O$ is the point $(0:0:1)$. One can define a binary operation $\oplus$ on $\mathcal{E}$ in such a way that $(\mathcal{E},\oplus)$ is an abelian group with identity element $O$, see, for example, Silverman \cite{Silverman2009}.

Explicitly this operation is defined in the following way. A {\em divisor} is a finite integer sum of points of $\mathcal{E}$. Suppose $g$ and $h$ are non-zero homogenous polynomials of the same degree in ${\mathbb R}[X_1,X_2,X_3]/(X_3^2X_1-X_2^3-aX_2X_1^2-bX_1^3)$. The divisor $(g/h)$ of $g/h$ is the sum of the points of $\mathcal{E}$ (over ${\mathbb C}$) which are zeros of $g$ minus the sum of points of $\mathcal{E}$ which are zeros of $h$, where intersections are counted with multiplicity. On the set of divisors we define an equivalence relation where $D \cong D'$ if and only if 
$$
D=D'+(g/h),
$$
for some $g$ and $h$. We define the sum $P \oplus Q$ as the point $T$ of $\mathcal{E}$ for which 
$$
P-O+Q-O \cong T-O.
$$
We can prove that this is well-defined. Let $h$ be the linear form whose kernel is the line joining $P$ and $Q$. There is a third point $R$ on $\mathcal{E}$ (counting with multiplicity) on the line joining $P$ and $Q$. Let $g$ be the linear form whose kernel is the line joining $R$ and $O$. The point $T$ is the third point on this line. Observe that
$$
(g/h)=T+R+O-(P+Q+R)=T+O-P-Q.
$$
Thus, we have proved that 
$$
P-O+Q-O+(g/h)=T-O,
$$
so $\oplus$ is well-defined. Moreover, this binary operation has the property that 
$$
P \oplus Q \oplus R=O
$$
if and only if $P$, $Q$ and $R$ are collinear.

Let $\phi$ denote the morphism $\mathcal{E} \rightarrow {\mathbb P}^4({\mathbb R})$ defined by
$$
\phi( (1:x:y))=(1:x:y:x^2:xy),
$$
and $\phi(O)=(0:0:0:0:1)$.

We can define a binary operation $\circ$ on $\phi(\mathcal{E})$ by 
$$
\phi(P) \circ \phi(Q)=\phi(P \oplus Q).
$$

\begin{lemma} \label{spansolid}
For distinct points $P$, $Q$, $R$ and $S$, the subspace spanned by $\phi(P),\phi(Q),\phi(R),\phi(S)$ is a solid.
\end{lemma}

\begin{proof}
Suppose that none of the points is the point $O$.
We need to show that the rank of the matrix
$$
\left(
\begin{array}{ccccc}
1 & x_1 & y_1 & x_1^2 & x_1y_1 \\
 1 & x_2 & y_2 & x_2^2 & x_2y_2 \\
 1 & x_3 & y_3 & x_3^2 & x_3y_3 \\
 1 & x_4 & y_4 & x_4^2 & x_4y_4 \\
\end{array}
\right)
$$
is four. If not then there are $a,b,c,d,e,f$ such that 
$$
y_i=a+bx_i+cx_i^2, \ \mathrm{and} \ xy_i=d+ex_i+fx_i^2,
$$
for all $i \in \{1,2,3,4\}$.
Eliminating $y$ from these equations we obtain a polynomial equation of degree $3$ satisfied by each $x_i$, $i \in \{1,2,3,4\}$. Thus,
$$
aX+bX^2+cX^3=d+eX+fX^2,
$$
which implies $a=e$, $b=f$ and $d=c=0$. This then implies that $Y=a+bX$ is a line incident with all four points $P$, $Q$, $R$ and $S$, contradicting the fact that a line is incident with at most three points of an elliptic curve.

Suppose that one of the points is the point $O$ which, without loss of generality, we can assume to be the point $P$. It is clear that the matrix
$$
\left(
\begin{array}{ccccc}
0 & 0 & 0 & 0 & 1 \\
 1 & x_2 & y_2 & x_2^2 & x_2y_2 \\
 1 & x_3 & y_3 & x_3^2 & x_3y_3 \\
 1 & x_4 & y_4 & x_4^2 & x_4y_4 \\
\end{array}
\right)
$$
has rank four, so the subspace spanned by $\phi(P),\phi(Q),\phi(R),\phi(S)$ is a solid.
\end{proof}

\begin{lemma} \label{spansolid2}
The group $(\phi(\mathcal{E}),\circ )$ is isomorphic to $(\mathcal{E},\oplus)$. The identity element of $(\phi(\mathcal{E}),+)$ is $\phi(O)$ and the binary operation has the property that
$$
\phi(P)\circ \phi(Q)\circ \phi(R)\circ \phi(S)\circ \phi(T)=\phi(O)
$$
if and only if $\phi(P)$, $\phi(Q)$, $\phi(R)$, $\phi(S)$ and $\phi(T)$ are contained in the same solid (hyperplane) of ${\mathbb P}^4({\mathbb R})$.
\end{lemma}

\begin{proof} (adapted from \cite[Proposition 3]{BGP2015}.)

Suppose that none of the points is the point $O$, so that $\phi(P)$, $\phi(Q)$, $\phi(R)$, and $\phi(S)$ are the points 
$$
(1:x_i:y_i:x_i^2:x_iy_i),
$$
$i\in \{1,2,3,4\}$, for some $x_i$ and $y_i$. 

By Lemma~\ref{spansolid}, these four points span a solid. If there is another point $(1:x:y:x^2:xy)$ on this solid then
$$
\left| \begin{array}{ccccc} 1 & x & y & x^2 & xy \\ 
1 & x_1 & y_1 & x_1^2 & x_1y_1 \\ 
1 & x_2 & y_2 & x_2^2 & x_2y_2 \\ 
1 & x_3 & y_3 & x_3^2 & x_3y_3 \\ 
1 & x_4 & y_4 & x_4^2 & x_4y_4 
\end{array}
\right|=0. 
$$
This gives
$$
y(x+c_3)=c_2 x^2+c_1x+c_0,
$$
for some $c_0,c_1,c_2,c_3$.  If this point is also a point of $\mathcal{E}$ then
$$
y^2=x^3+ax+b.
$$
Eliminating $y$ from these two equations we get a polynomial of degree $5$ in $x$. There are four real solutions $x_1,x_2,x_3,x_4$ and the fifth real solution gives us the point $\phi(T)$. 

Suppose that $\phi(P)$, $\phi(Q)$, $\phi(R)$, $\phi(S)$ and $\phi(T)$ are contained in the same solid (hyperplane) of ${\mathbb P}^4({\mathbb R})$.
Then there is a quadratic form $g=g(X_1,X_2,X_3)$ defining a conic containing $P$, $Q$, $R$, $S$ and $O$. Since the above determinant is zero this conic necessarily contains $T$. Let $h$ be the linear form whose kernel is the hyperplane $X_1=0$. Then the divisor
$$
(g/h^2)=P +Q+ R + S+ T + O-6O
$$
so
$$
P-O +Q-O+ R-O + S-O+ T-O 
$$
is zero in the quotient of the divisors, which implies
$$
P \oplus Q \oplus R \oplus S \oplus T=O,
$$
and hence
$$
\phi(P)\circ \phi(Q)\circ \phi(R)\circ \phi(S)\circ \phi(T)=\phi(O).
$$
A similar argument works in the case that one of the points is the point $O$.

To prove the converse, suppose that
$$
\phi(P)\circ \phi(Q)\circ \phi(R)\circ \phi(S)\circ \phi(T)=\phi(O).
$$
Let $g=g(X_1,X_2,X_3)$ be the quadratic form defining the conic containing $P$, $Q$, $R$, $S$ and $O$ and let $h$ be the linear form whose kernel is the hyperplane $X_1=0$. Then
$$
(g/h^2)=P +Q+ R + S+ U + O-6O
$$
for some $U$ in the intersection of the conic and $\mathcal{E}$. Since, as divisors
$$
P +Q+ R + S+ T \cong 5O,
$$
we conclude that $T \cong U$. Hence, $T$ is on the conic defined by $g$. 

Coordinatising the points $\phi(P)$, $\phi(Q)$, $\phi(R)$, $\phi(S)$ and $\phi(T)$ as 
$$
(1:x_i:y_i:x_i^2:x_iy_i),
$$
for $i\in \{1,2,3,4,5\}$,
this implies that
$$
\left(
\begin{array}{cccccc}

1 & x_1 & y_1 & x_1^2 & x_1y_1 & y_1^2\\
 1 & x_2 & y_2 & x_2^2 & x_2y_2 & y_2^2\\
 1 & x_3 & y_3 & x_3^2 & x_3y_3 & y_3^2\\
 1 & x_4 & y_4 & x_4^2 & x_4y_4 & y_4^2\\
  1 & x_5 & y_5 & x_5^2 & x_5y_5& y_5^2 \\
  0 & 0 & 0 & 0 & 0 & 1 \\
\end{array}
\right)
$$
has rank $5$. Hence, the top left $5 \times 5$ submatrix has rank $4$, which implies that $\phi(P)$, $\phi(Q)$, $\phi(R)$, $\phi(S)$ and $\phi(T)$ are contained in the same solid of ${\mathbb P}^4({\mathbb R})$.
\end{proof}

Suppose $\mathcal{S}$ is a set of $n \geqslant 6$ points selected as a subgroup of $(\mathcal{E},\oplus)$. By Lemma~\ref{spansolid}, any four points of $\phi(\mathcal{S})$ span a solid and, by Lemma~\ref{spansolid2}, $\phi(\mathcal{S})$ spans ${\mathbb P}^4({\mathbb R})$. The number of ordinary solids spanned by $\phi(\mathcal{S})$ will be the number of solutions to
$$
2P \oplus Q \oplus R \oplus S=0,
$$
where $P$, $Q$, $R$ and $S$ are distinct points of $\mathcal{S}$. This, we can bound from by above by choosing $P$ and then choosing a $2$-subset $\{ Q,R\}$ of points of $\mathcal{S} \setminus \{P\}$. Then $\mathcal{S}$ spans at most
$$
\tfrac{1}{2}n (n-1) (n-2)
$$
ordinary solids. If $\mathcal{S}$ is $2$-torsion free then we can choose $\{ Q, R, S\}$
and bound the number of ordinary solids by
$$
\tfrac{1}{6}n (n-1) (n-2).
$$

Therefore, if the group of the elliptic curve has a $2$-torsion free subgroup of size $n$ then we can construct a set $\mathcal{S}$ of points in ${\mathbb P}^4({\mathbb R})$ with the property that no four points of $\mathcal{S}$ are co-planar, the set $\mathcal{S}$ spans the whole space and spans at most 
$$
\tfrac{1}{6}n (n-1) (n-2)
$$
ordinary solids.

The set of points of $\phi(\mathcal{E})$ is contained in the following five linearly independent quadrics,
$$
X_4X_1-X_2^2,\\
$$
$$
X_2X_3-X_5X_1,\\
$$
$$
X_3X_4-X_5X_2,\\
$$
$$
X_3^2-X_2X_4-aX_1X_2-bX_1^2,\\
$$
$$
X_3X_5-X_4^2-aX_4X_1-bX_1X_2.
$$
Our aim will be to prove a converse of this, that if the number of ordinary solids spanned by $\mathcal{S}$ is small then all but a few points of $\mathcal{S}$ are contained in the intersection of five linearly independent quadrics.

Apart from the elliptic curve example, we also have a trivial example which spans few ordinary solids. If we take $\mathcal{S}$ to be the a point $P$ union the set of $n-1$ points on a normal rational curve in a hyperplane not incident with $P$, then $\mathcal{S}$ spans at most
$$
\tfrac{1}{6}(n-1) (n-2)(n-3)
$$
ordinary solids.

\section{The graph associated with a set of hyperplanes}

Let $\mathcal{S}$ be a set of $n$ points in ${\mathbb P}^4({\mathbb R})$ with the property that any four points of $\mathcal{S}$ span a solid. We begin by defining the graph $\Gamma=\Gamma(S)$ associated with $\mathcal{S}$. In the dual space the points of $\mathcal{S}$ are hyperplanes and we denote this set of hyperplanes by $\mathcal{S}^*$. Let $p^*$ denote the solid in the dual space, dual to the point $p$. Let $p,q,r,s$ be four points of $\mathcal{S}$. Since $p$, $q$, $r$ and $s$ span a solid, in the dual space $p^* \cap q^*\cap r^* \cap s^*$ is a point. The point $p^* \cap q^*\cap r^* \cap s^*$ is defined to be a vertex of the graph. Observe that we do not rule out the possibility that there are more points of $\mathcal{S}$ whose dual is incident with this point. Indeed, with our additional hypothesis on $\mathcal{S}$ this will be the norm, rather than the exception. For each $p,q,r \in \mathcal{S}$, the edges of the graph $\Gamma$ are line segments on the line $p^*\cap q^*\cap r^*$ cut out by the vertices. For each $p,q \in \mathcal {S}$, the faces of the graph $\Gamma$ are the faces cut out by the edges in the plane $p^* \cap q^*$.

Following Green and Tao's use of a triangular grid in ${\mathbb P}^2({\mathbb R})$ in \cite{GT2013} and the first author's generalisation to ${\mathbb P}^3({\mathbb R})$ in \cite{Ball2018}, we further generalise these definitions to ${\mathbb P}^4({\mathbb R})$.

 Let $I, J, K, L, M$ be discrete intervals in $\mathbb{Z}$. A {\em $5$-cell grid} of
 dimension $$
 |I| \times |J| \times |K| \times  |L| \times  |M|
 $$ 
 is a collection of hyperplanes $(p_i^*)_{i \in I}$, 
 $(q_j^*)_{j \in J}$, $(r_k^*)_{k \in K}$, $(s_l^*)_{l \in L}$, $(t_m^*)_{m \in M}$
 in $\mathbb{P}^4({\mathbb R})$ such that $p_i^*$, $q_j^*$, $r_k^*$, $s_l^*$, $t_m^*$ intersect
 in a point if and only if $i + j + k + l + m = 0$ and no other hyperplane of the 
 grid passes through the point of intersection.

Along with the notion of $5$-cell grid, we introduce the notion of
good/bad edges. In their paper, Green and Tao defined good
and bad edges in order to characterise edges of the graph $\Gamma$ according to
whether they contributed to the triangular grid structure of the graph or not. \\

According to Green and Tao \cite{GT2013}, a good edge in the graph in $\mathbb{P}^2({\mathbb R})$ was an edge that was the side of two triangular faces and whose end vertices were each of
degree $6$. A bad edge was simply any edge that was not good.

In ${\mathbb P}^4({\mathbb R})$ we will define a good edge in the following way. A {\em good edge} of the graph ${\mathbb P}^4({\mathbb R})$ is any edge $e$ in $\Gamma$ such
 that any face that contains $e$ is a triangle and such that the end vertices of
 $e$ are each incident with exactly $5$ hyperplanes of $S^*$. A {\em bad edge} is any edge
 of $\Gamma$ that is not good.

To state this in the same way as Green and Tao we could say that the end vertices
of a good edge have degree $20$.

We say an edge of the graph $\Gamma$ is a {\em rather good edge} if it is a good edge and every edge coming from its end vertices is also a good edge. We will call an edge of $\Gamma$ {\em slightly bad} if it is not a rather good edge.

These two concepts are important, as their appearance comes naturally inside the
$5$-cell structure. The other implication is also true, that
the structure of $\Gamma$ around a rather good edge is that of a $5$-cell grid.
This fact will be more useful to us, and we will prove it now.

%The following lemma is from Ball \cite[Lemma 5]{Ball2018}.

%\begin{lemma}
% Let $e$ be a rather good edge of the graph $\Gamma$ lying on the line $p^* \cap
% q^* \cap r^*$. If there is a triangle in the plane $p^* \cap q^*$ with side
% $e$ that has its other sides cut out by $s^*$ and $t^*$, then there are
 %triangles in the planes $p^* \cap r^*$ and $q^* \cap r^*$ with side $e$ whose
% sides are also cut by $s^*$ and $t^*$ respectively.
%\end{lemma}

\begin{lemma} \label{rathergoodedge}
 The structure of $\Gamma$ around a rather good edge $e$ on the line $p_0^* \cap q_0^* \cap t_0^*$ cut on one end by $s_0^*$ and $r_0^*$ and the other end by $s_1^*$ and $r_{-1}^*$ is that of a $5$-cell grid where the hyperplanes $p_i^*$, $q_j^*$, $t_k^*$, $s_l^*$, $r_m^*$ intersect in a point if and only if $i + j + k + l + m = 0$ for all $i,j,k \in \{-1,0,1\}$, $l \in \{0,1\}$ and $m \in \{-1,0\}$ unless (possibly) $\{i,j,k\} \subseteq \{-1,1\}$.
\end{lemma}

\begin{proof}
%Let $e$ be a rather good edge on the line $p_0^* \cap q_0^* \cap t_0^*$ cut on one end by $s_0^*$ and $r_0^*$ and the other end by $s_1^*$ and $r_{-1}^*$. 

Since $e$ is a rather good edge, the edge along the line $p_0^* \cap q_0^* \cap r_0^*$, cut out by $p_0^* \cap q_0^* \cap r_0^*\cap s_0^* \cap t_0^*$ and $p_0^* \cap q_0^* \cap r_0^*\cap s_1^*$ is a good edge, so there is another point, $t_{-1} \in \mathcal S$, such that $t_{-1}^*$ is incident with the point $p_0^* \cap q_0^* \cap r_0^*\cap s_1^*$.

Similarly, the edge along the line $p_0^* \cap q_0^* \cap r_{-1}^*$, cut out by $p_0^* \cap q_0^* \cap r_{-1}^*\cap s_1^* \cap t_0^*$ and $p_0^* \cap q_0^* \cap r_{-1}^*\cap s_0^*$ is a good edge, so there is another point, $t_{1} \in \mathcal S$, such that $t_{1}^*$ is incident with the point $p_0^* \cap q_0^* \cap r_{-1}^*\cap s_0^*$.

By Ball \cite[Lemma 4.1]{Ball2018}, in each of the planes $p_0^* \cap q_0^*$, $p_0^* \cap t_0^*$ and $t_0^* \cap q_0^*$, the edge $e$ is in a triangle whose other two sides are cut out by (after a possible relabelling) $s_0^*$ and $r_{-1}^*$ on one side and $r_0^*$ and $s_1^*$ on the other. In other words, the triangles are cut out by the same pair of hyperplanes in each of the three planes.

Thus, the points $p_0$, $q_0$ and $t_0$ are interchangeable in the above argument, so we have that there are points $p_{1},p_{-1},q_{1},q_{-1}$ in $\mathcal S$ which are incident with the points $t_0^* \cap q_0^* \cap r_{-1}^*\cap s_0^*$, $t_0^* \cap q_0^* \cap r_0^*\cap s_1^*$, $t_0^* \cap p_0^* \cap r_{-1}^*\cap s_0^*$, $t_0^* \cap p_0^* \cap r_0^*\cap s_1^*$, respectively.

According to Ball \cite[Lemma 4.2]{Ball2018}, in the solid $t_0^*$, the structure around the rather good edge $e$ is a double diamond, Figure~\ref{doublediamond}. By reading off the points in Figure~\ref{doublediamond}, we see that the hyperplanes $p_i^*$, $q_j^*$, $t_k^*$, $s_l^*$, $r_m^*$ intersect in a point if and only if $i + j + k + l + m = 0$ for all $i,j \in \{-1,0,1\}$, $k=0$, $l \in \{0,1\}$ and $m \in \{-1,0\}$.

\begin{figure}[h]
\centering
\includegraphics[width=6 in]{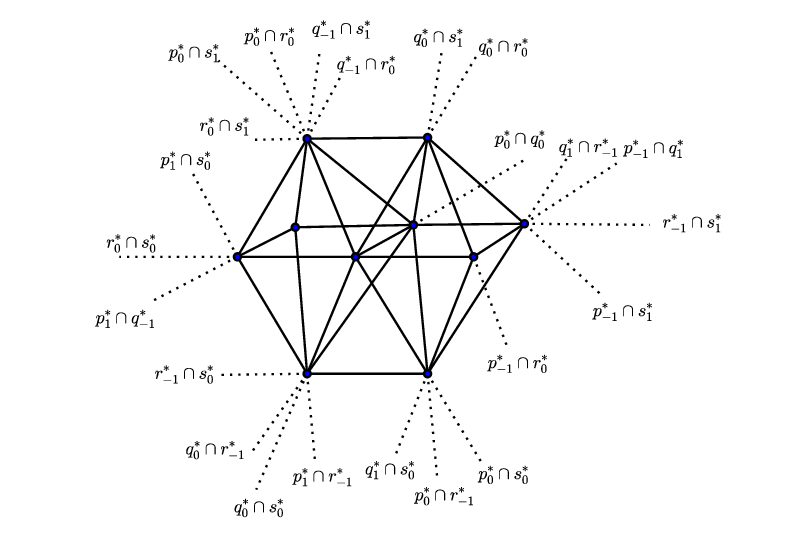}
\caption{The double diamond structure around a rather good edge in the solid $t_0^*$. Each line should be read as intersect $t_0^*$, so, for example, the line $p_1^*\cap q_{-1}^*$ in the figure is the line $p_1^*\cap q_{-1}^*\cap t_0^*$.}
\label{doublediamond}
\end{figure}

Applying the same lemma to $p_0^*$ and $q_0^*$, we conclude that the hyperplanes $p_i^*$, $q_j^*$, $t_k^*$, $s_l^*$, $r_m^*$ intersect in a point if and only if $i + j + k + l + m = 0$ for all $i,j,k \in \{-1,0,1\}$, $l \in \{0,1\}$ and $m \in \{-1,0\}$ unless $\{i,j,k\} \subseteq \{-1,1\}$.
\end{proof}

We have already mentioned the relationship between good and rather good edges and
$5$-cell grids. The objective of introducing
the grids in the first place was the idea that sets of points spanning
few ordinary hyperplanes should be rich in hyperplanes containing five points of $S$, and thus, should
contain $5$-cell grids.

One of the first results we can prove that indicates that this is the case is Lemma~\ref{bedges}, which indicates there are few bad edges in the graph $\Gamma$.

\begin{lemma} \label{bedges}
Let $\mathcal{S}$ be a set of $n$ points in $\mathbb{P}^4({\mathbb R})$ not all in a hyperplane and such
that any $4$ points of $\mathcal{S}$ spans a hyperplane.  If $\mathcal{S}$ spans less
than $Kn^3$ ordinary hyperplanes then the graph $\Gamma = \Gamma(\mathcal{S})$ has at most $48Kn^3$ bad edges.
\end{lemma}

\begin{proof}
Let $V$ and $E$ be the set of vertices and edges of the graph $\Gamma$ respectively.
 
 Let $\pi$ and $\pi'$ be two elements of $\mathcal{S}^*$. As these are hyperplanes of 
 $\mathbb{P}^4({\mathbb R})$, the intersection $\pi \cap \pi'$ is a plane.
 
 Let us call $V_{\pi,\pi'}$ the set of vertices of $\Gamma$ which are in both
 $\pi$ and $\pi'$, and let us define in the same way the set $E_{\pi, \pi'}$ of
 edges in the intersection. Let us also consider the set of faces $F_{\pi, \pi'}$
 in the plane $\pi \cap \pi'$ defined by these vertices and edges. As $\pi \cap \pi'$
 is a projective plane, Euler's formula is
 
 \begin{equation} \label{Eu2}
  |V_{\pi, \pi'}| - |E_{\pi, \pi'}| + |F_{\pi, \pi'}| = 1.
 \end{equation}
 
 Let us call $F$ the set of faces in the graph $\Gamma$, which we will take as 
 the union of the faces defined in the planes $\pi \cap \pi'$ for all 
 $\pi, \pi' \in \mathcal{S}^*$, so
 
 \begin{equation*}
  F = \bigcup_{\pi, \pi' \in \mathcal{S}^*} F_{\pi, \pi'}.
 \end{equation*}
 
 Now, taking the Euler's formula \eqref{Eu2} into account, if we sum for all the
 choices of $\pi \cap \pi'$ in $\mathcal{S}^*$ we end up with
 
 \begin{equation} \label{Eu3}
  \sum_{i = 4}^n \frac{i(i-1)}{2}v_i - 3|E| + |F| = \frac{n(n-1)}{2},
 \end{equation}

 where the numbers $v_i$ represents the number of vertices in $\Gamma$ which are
 incident with exactly $i$ hyperplanes of $\mathcal{S}^*$. This formula follows since any face is contained in exactly one plane $\pi \cap \pi'$, a line in $\mathbb{P}^4({\mathbb R})$ is the intersection of $3$ hyperplanes (thus, any edge of $E$ is in exactly $3$ planes $\pi \cap \pi'$), and a vertex which is incident with exactly $i$ hyperplanes of $\mathcal{S}^*$ will be contained in $\frac{i(i-1)}{2}$ planes $\pi \cap \pi'$.

% Now we will use some counting arguments to derive some relations between the number of faces, edges and vertices of the graph.
 
Since each edge is contained in exactly three planes $\pi \cap \pi'$ and since in each of these planes, it lies between two faces, we obtain the following equation by counting pairs $(e, f)$ of edges and faces in two different ways,
 
 \begin{equation} \label{edgeface}
  \#(e,f) = 6|E| = \sum_{j \geqslant 3} jf_j,
 \end{equation}
 
 where $f_j$ denotes the number of faces in $F$ with $j$ edges.\\
 
Now we count pairs $(v, e)$ of edges and vertices in two different ways.
Note that for a vertex incident with exactly $i$ hyperplanes
 there are $\frac{i(i-1)(i-2)}{6}$ lines (which are intersection of three hyperplanes)
 incident with it. Therefore,
 
 \begin{equation} \label{vertexedge}
%  \begin{aligned}
 2|E| = \sum_{i \geqslant 4} 2\times\frac{i(i-1)(i-2)}{6} v_i. \\
%		   &\Downarrow \\
%	      6|E| &= \sum_{i \geqslant 4} i(i-1)(i-2)v_i
%  \end{aligned}
 \end{equation}
 
 Now we will use these equations and the fact that $v_4 \leqslant Kn^3$ (since $v_4$
 represents the number of ordinary hyperplanes), to obtain the bound for the number
 of bad edges in $\Gamma$.
 
 We can combine the Euler's formula \eqref{Eu3} with the equations \eqref{edgeface} and
 \eqref{vertexedge} to obtain the following:
 
 \begin{equation*}
  \begin{aligned}
   |F| - 3|E| + \sum_{i \geqslant 4}\frac{i(i-1)}{2}v_i = \frac{n(n-1)}{2} ,	\\
   6|F| - 18|E| + 3\sum_{i \geqslant 4}i(i-1)v_i = 3n(n-1) 			,\\
   6|F| - \left[2\sum_{j \geqslant 3}jf_j\right] -
   \left[\sum_{i \geqslant 4}i(i-1)(i-2)v_i\right] + 
   3\sum_{i \geqslant 4}i(i-1)v_i = 3n(n-1) 					,\\
   -2\sum_{j \geqslant 3}(j-3)f_j - \sum_{i \geqslant 4}i(i-1)(i-5)v_i = 3n(n-1),
  \end{aligned}
 \end{equation*}
 \begin{equation} \label{v4_1}
  12v_4 = 3n(n-1) + 2\sum_{j \geqslant 4}(j-3)f_j + \sum_{i \geqslant 6}i(i-1)(i-5)v_i.
 \end{equation}
 
 We can use this last equation \eqref{v4_1} to obtain some bounds, since every
 summand on the right hand side is non-negative. 
 
 First we deduce that for the faces  
 
 \begin{equation*} 
  \begin{aligned}
   \sum_{j \geqslant 4}(j-3)f_j &\leq 6v_4 &\leq 6Kn^3 ,		\\
   \sum_{j \geqslant 4}f_j &\leq \sum_{j \geqslant 4}(j-3)f_j &\leq 6Kn^3 ,\\
     \end{aligned}
     \end{equation*}
   \begin{equation}  \label{fromface}
   \sum_{j \geqslant 4}jf_j \leq 6Kn^3 + 18Kn^3 = 24Kn^3.
 \end{equation}

 This gives us a bound for the number of edges incident with a face which is not a 
 triangle. Now, if we look at the vertices:
 
 \begin{equation*}
  \begin{aligned}
   \sum_{i \geqslant 6}i(i-1)(i-5)v_i &\leq 12v_4 &\leq 12Kn^3			\\
   \sum_{i \geqslant 6}i(i-1)v_i &\leq \sum_{i \geqslant 6}i(i-1)(i-5)v_i &\leq 12Kn^3	\\
   \sum_{i \geqslant 6}i(i-1)(i-2)v_i &\leq 12Kn^3 + 36Kn^3 &= 48Kn^3.
  \end{aligned}
 \end{equation*}
 
 From this last equation we can bound for the number of edges incident with a
 vertex which is not incident with exactly $5$ hyperplanes,
 
 \begin{equation} \label{fromvert}
  \sum_{i \neq 5}\frac{i(i-1)(i-2)}{3}v_i = 8v_4 + \sum_{i \geqslant 6}\frac{i(i-1)(i-2)}{3}v_i
  \leq 8Kn^3 + \frac{48Kn^3}{3} = 24Kn^3.
 \end{equation}
 
 Now, an edge can be bad if either it is contained in some non-triangular face or if
 one of its end vertices is not incident with exactly $5$ hyperplanes (i.e. the vertex
 has degree different than $20$). Therefore, if we join the two equations
 \eqref{fromface} and \eqref{fromvert}, we obtain the following bound
 
 \begin{equation}
  \#\textrm{Bad Edges} \leq 24Kn^3 + 24Kn^3 = 48Kn^3.
 \end{equation}
\end{proof}

\begin{lemma} \label{sbedges}
 The number of slightly bad edges is at most $1872Kn^3$.
%  Let $S$ be a set of $n$ points in $\mathbb{P}^4(\mathbb{R})$ not all in an hyperplane and such
%  that any subset of $4$ points span a hyperplane. Let us assume that $S$ spans less
%  than $Kn^3$ ordinary hyperplanes. Then in the graph $\Gamma$ there are at most $1872Kn^3$
%  slightly bad edges.
\end{lemma}
\begin{proof}
 This is quickly proven by counting. A slightly bad edge is either a bad edge itself
 or a good edge which is incident with a bad edge. We know from Lemma \ref{bedges}
 that the number of bad edges in $\Gamma$ is at most $48Kn^3$. On the other hand,
 we know that the end vertices of a good edge have degree $20$ (since they are
 incident to $5$ hyperplanes of $\mathcal{S}^*$). Thus, we can bound the number of slightly
 bad edges by the following rough estimation:
 
 \begin{equation}
  \#\textrm{Slightly Bad Edges} \leq 48Kn^3 + 2 \times (20-1) \times 48Kn^3 = 1872Kn^3.
 \end{equation}

\end{proof}

\section{Intersection of five quadrics}

We have defined in the previous section the concepts of $5$-cell grid and good/bad
edges and how they relate to our problem. In keeping with the work
of Green and Tao \cite{GT2013}, we now study if sets of points which dualise into $5$-cell
grids are contained in a certain variety and, in that case, deduce what is the nature of
that variety.

The answer to that question is that, indeed, sets of points which dualise in
$\mathbb{P}^4({\mathbb R})$ into $5$-cell grids are contained the intersection of $5$ quadrics.

In this section we will prove this, along with some other results that we will need
later about these varieties. 

We say that a quadric is {\em degenerate} at a point $p$, if for every point $x$ on the quadric the entire line joining $p$ and $x$ is contained in the quadric. Equivalently, if $p$ is the point $(1:0:0:0:0)$ with respect to some basis, the quadratic form defining the quadric has no $X_1$ term.
 
Before studying the characteristics of the variety, we prove that the sets of points forming the $5$-cell grid are contained in the intersection of five quadrics. For this we will need two lemmas, the first of which is the following.

\begin{lemma} \label{11pts}
 Let $S=\{p_0, p_1, q_0, q_1, r_{-1}, r_0, s_{-1}, s_0, t_{-1}, t_0, t_1\}$ be eleven
 points of $\mathbb{P}^4({\mathbb R})$ such that $p_i, q_j, r_k, s_l, t_m$ are contained in a
 hyperplane if and only if $i + j + k + l + m = 0$. Then there are five linearly
 independent quadrics that contain the eleven points of $S$ and six linearly
 independent quadrics that contain the ten points of $S \setminus \{t_0\}$.
\end{lemma}
\begin{proof}
We can choose a basis so that 
 \begin{equation*}
  \begin{aligned}
  p_0&=(1:0 :0:0:0) ,\ \ p_1&=&(a_1:a_2:a_3:a_4:a_5), \\
q_0&=(0:1:0:0:0),\ \ q_1&=&(b_1:b_2:b_3:b_4:b_5), \\
r_{-1}&=(0 :0 :1:0:0),\ \ r_{0}&=&(c_1:c_2:c_3:c_4:c_5), 
\\
s_{-1}&=(0:0:0:1:0),\ \ s_{0}&=&(d_1:d_2:d_3:d_4:d_5), 
\\
t_0&=(0:0:0:0:1). \ \ & &
  \end{aligned}
 \end{equation*}
 
 Each of the following subsets of five points are contained in a hyperplane 
 $$
 \{p_0,q_0,r_0,s_0,t_0\}, \{p_1,q_0,r_{-1},s_0,t_0\} , \{p_1,q_0,r_0,s_{-1},t_0\}, 
  $$
  $$
 \{p_1,q_1,r_{-1},s_{-1},t_0\}, \{p_0,q_1,r_{0},s_{-1},t_0\} , \{p_0,q_1,r_{-1},s_{0},t_0\} ,
  $$
  which implies that
  $$
  d_4c_3=c_4d_3, \ b_1a_2=a_1b_2, \ a_4d_1=a_1d_4,
  $$
    $$
  b_3c_2=c_3b_2, \ a_3c_1=c_3a_1, \ d_4b_2=b_4d_2.
  $$
  Since $t_1$ is the intersection of the hyperplanes spanned by the following four sets of points
  $$
   \{p_1,q_0,r_{-1},s_{-1} \}, \{p_0,q_0,r_{0},s_{-1}\} , \{p_0,q_0,r_{-1},s_{0}\},  \{p_0,q_1,r_{-1},s_{-1}\}, 
  $$
  we can calculate the coordinates of $t_1$. Likewise, we can calculate the coordinates of the points $t_{-1}$, since it is the intersection of the hyperplanes spanned by the four sets of points
  $$
     \{p_0,q_{1},r_{0},s_{0} \}, \{p_1,q_1,r_{-1},s_{0}\} , \{p_1,q_1,r_{0},s_{-1}\},  \{p_1,q_0,r_{0},s_{0}\}.
  $$
 
Let $\mathrm{M}$ denote the $11 \times 15$ matrix whose rows are indexed by the points of $S$ and whose columns are indexed by pairs $(i,j)$ where $1 \leqslant i \leqslant j \leqslant 5$. The entry in the row indexed by $x$ and the column indexed by the pair $(i,j)$ is $x_ix_j$.

With the aid of  a computer algebra package, we deduce that the matrix $\mathrm{M}$ has rank at most $10$. Therefore, the right-kernel of $\mathrm{M}$ has dimension at least $5$, which corresponds to the quadrics which are zero at the points of $S$. If we remove the row corresponding to the point $t_0$, the revised matrix has rank at most $9$, so the right-kernel has dimension at least $6$, which corresponds to the quadrics which are zero at the points of $S \setminus \{t_0\}$.
\end{proof}

Notice that the space of quadrics in $\mathbb{P}^4({\mathbb R})$ has dimension ${6 \choose 2} = 15$ and that, in a generic situation, each point we force our quadrics to contain 
should impose a linearly independent condition on the space of quadrics. Thus, 
in that general situation we would expect the space of quadrics passing through $11$ 
points to be of dimension $4$ and the space of quadrics passing through $10$ points 
to be of dimension $5$. However, in Lemma \ref{11pts} we proved that there are $6$ linearly
independent quadrics passing through the $10$ points used in the proof, and $5$
quadrics through the whole set of $11$ points. This constitutes one more dimension
than what we would expect in a general situation.\\

From this lemma we obtain the variety that we were looking for, the intersection
of $5$ linearly independent quadrics. In order to extend this result to the entire 
$5$-cell grid, we will need a further result, which is closely related to the previous proof.

\begin{lemma} \label{10points}
 Let $\{p_0, p_1, q_0, q_1, r_{-1}, r_0, s_{-1}, s_0, t_{-1}, t_1\}$ be $10$ points
 on $\mathbb{P}^4(\mathbb{R})$ such that $p_i,$ $q_j$, $r_k$, $s_{l}$, $t_m$ are contained in a
 hyperplane if and only if $i + j + k + l + m = 0$. Then any quadric that contains
 $9$ of the $10$ points must also contain the tenth one.
\end{lemma}
\begin{proof}
 As mentioned above, the expected dimension of the space of quadrics passing through
 the $10$ points should be $5$, instead of having the $6$ linearly independent
 quadrics we proved in Lemma \ref{11pts}.
 
We want to prove that the space of quadrics through any subset of $9$ points
 of these $10$ has dimension $6$, as one would expect. This will imply that
 the $6$ linearly independent quadrics passing through any $9$ points are the same
 as the ones going through the $10$ points, and thus, any quadric going through
 $9$ points will contain the tenth point of this set.
 
 For this we prove that any subset of $9$ points of these impose
 $9$ linearly independent conditions on the space of quadrics. \\
 
 Without loss of generality, let us suppose that the subset of $9$ points we
 choose is the one without the point $t_1$. We can choose an appropriate basis so that
 the points have the following coordinates:
 
 \begin{equation*}
  \begin{aligned}
  p_0&=(1:0 :0:0:0) ,\ \ p_1&=&(a_1:a_2:a_3:a_4:a_5), \\
q_0&=(0:1:0:0:0),\ \ q_1&=&(b_1:b_2:b_3:b_4:b_5), \\
r_0&=(0 :0 :1:0:0),\ \ r_{-1}&=&(c_1:c_2:c_3:c_4:c_5), 
\\
s_{0}&=(0:0:0:1:0),\ \ s_{-1}&=&(d_1:d_2:d_3:d_4:d_5), 
\\
t_{-1}&=(0:0:0:0:1). \ \ & &
  \end{aligned}
 \end{equation*}
 
 %Now let us consider $\phi$ a generic quadric in $\mathbb{P}^4(\mathbb{R})$,
  
% \begin{equation*}
 % \phi(X) = \alpha_{11}X_1^2 + ... + \alpha_{55}X_5^2 + \alpha_{12}X_1X_2 +
%	    ... + \alpha_{45}X_4X_5.
 %\end{equation*}
 
% To impose that the quadric $\phi$ contains a certain point $x$ is the same as
 %to impose the equation $\phi(x) = 0$. Now, if we want to impose that this
% quadric contains the $9$ points mentioned above, we obtain the following
 %system of equations:
 
 As in the proof of Lemma~\ref{11pts}, the quadrics which are zero on 
 $$
 \{p_0, p_1, q_0, q_1, r_{-1}, r_0, s_{-1}, s_0, t_{-1}\}
 $$
 are given by the non-zero elements of the right-kernel of the matrix
 \begin{equation*}
\mathrm{M}=\left(  \begin{array}{ccccccccccccccc}
   1 & 0 & 0 & 0 & 0 & 0 & 0 & 0 & 0 & 0 & 0 & 0 & 0 & 0 & 0 \\
   0 & 1 & 0 & 0 & 0 & 0 & 0 & 0 & 0 & 0 & 0 & 0 & 0 & 0 & 0 \\
   0 & 0 & 1 & 0 & 0 & 0 & 0 & 0 & 0 & 0 & 0 & 0 & 0 & 0 & 0 \\
   0 & 0 & 0 & 1 & 0 & 0 & 0 & 0 & 0 & 0 & 0 & 0 & 0 & 0 & 0 \\
   0 & 0 & 0 & 0 & 1 & 0 & 0 & 0 & 0 & 0 & 0 & 0 & 0 & 0 & 0 \\
   a_1^2 & a_2^2 & a_3^2 & a_4^2 & a_5^2 & a_1a_2 & a_1a_3 & a_1a_4 & a_1a_5
   & a_2a_3 & a_2a_4 & a_2a_5 & a_3a_4 & a_3a_5 & a_4a_5 \\
   b_1^2 & b_2^2 & b_3^2 & b_4^2 & b_5^2 & b_1b_2 & b_1b_3 & b_1b_4 & b_1b_5
   & b_2b_3 & b_2b_4 & b_2b_5 & b_3b_4 & b_3b_5 & b_4b_5 \\
   c_1^2 & c_2^2 & c_3^2 & c_4^2 & c_5^2 & c_1c_2 & c_1c_3 & c_1c_4 & c_1c_5
   & c_2c_3 & c_2c_4 & c_2c_5 & c_3c_4 & c_3c_5 & c_4c_5 \\
   d_1^2 & d_2^2 & d_3^2 & d_4^2 & d_5^2 & d_1d_2 & d_1d_3 & d_1d_4 & d_1d_5 
   & d_2d_3 & d_2d_4 & d_2d_5 & d_3d_4 & d_3d_5 & d_4d_5
  \end{array}\right).
 \end{equation*}
 
If we can prove the matrix $\mathrm{M}$ has rank $6$ then the space of quadrics which are zero on the nine points has dimension $6$.
 
 The five points corresponding to the basis of $\mathbb{P}^4(\mathbb{R})$ implies immediately that the matrix has
 rank at least $5$. We only need to prove that the $4\times10$ sub-matrix
 corresponding to the last four rows has rank $4$.
 
 The key ingredient to prove this comes from the nature of the set we are
 dealing with. We will use the fact that the points $p_i, q_j, r_k, s_{l}, t_m$ are in a hyperplane if and only
 if $i + j + k + l + m = 0$. \\
 
 First of all, we can suppose that none of the coefficients $a_5, b_5, c_5, d_5$
 are $0$. This is because, if any of them were, that point would be contained in
 the hyperplane spanned by $p_0, q_0, r_0, s_0$ and by hypothesis they are not.
 
 Secondly, we know that the $5$ points $p_1, q_1, r_{-1}, s_{-1}$ and $t_{-1}$
 are not contained in an hyperplane, as their indices sum to $-1$. This means
 that the $5$ points are linearly independent and that the matrix
 
$$
  \left(\begin{array}{ccccc}
    0 & 0 & 0 & 0 & 1 \\
    a_1 & a_2 & a_3 & a_4 & a_5 \\
    b_1 & b_2 & b_3 & b_4 & b_5 \\
    c_1 & c_2 & c_3 & c_4 & c_5 \\
    d_1 & d_2 & d_3 & d_4 & d_5
  \end{array}\right)
$$

 has rank $5$.
 
 Now, since we know that the coefficients $a_5, b_5, c_5, d_5$ are different
 from $0$, we can scale the points to make them $1$ and get rid of the $X_5$
 coefficients. After doing that, looking at the sub-matrix given by the columns
 corresponding to the coefficients of $X_1X_5$, $X_2X_5$, $X_3X_5$ and $X_4X_5$,
 we get the following matrix
 
$$
  \left(\begin{array}{ccccc}
    a_1 & a_2 & a_3 & a_4 \\
    b_1 & b_2 & b_3 & b_4 \\
    c_1 & c_2 & c_3 & c_4 \\
    d_1 & d_2 & d_3 & d_4 
  \end{array}\right)
$$
 
 which we know has rank $4$. 
 
 Hence, the rank of $\mathrm{M}$ is $9$. 
 
 If the point left out was not $t_1$ but one of the other points, we would need to use a different basis, but we would be able to apply the same reasoning and reach the same conclusion.
 
Since we have proven that there are $6$ linearly independent quadrics through
 each subset of $9$ points, we conclude that any quadric passing through a
 subset of $9$ points also contains the tenth point.
\end{proof}

In Lemma \ref{11pts} and Lemma \ref{10points} we have proved that some small
sets of points on a $5$-cell grid are contained in the intersection of $5$ linearly
independent quadrics. What we want to prove now, using these two results, is that 
the whole $5$-cell grid structure must be contained in the intersection of $5$ 
quadrics, which is the main theorem of this section.

In order to use these lemmas later, we will need to be careful and verify that the points
to which we apply them maintain the same structure as the points in the lemma. With this in mind, and to avoid complications later, we will present the sets of points in these
two lemmas as follows.

We will apply Lemma \ref{11pts} to the following set of 11 points.

\begin{table}[h!]
 \centering
 \begin{tabular}{|c|c|c|c|c|}
  \hline
   $p_1$ & $q_1$ &          &          & $t_1$    \\
  \hline
   $p_0$ & $q_0$ & $r_0$    & $s_0$    & $t_0$	  \\
  \hline
         &       & $r_{-1}$ & $s_{-1}$ & $t_{-1}$ \\
  \hline
 \end{tabular}
\end{table}

We will apply Lemma \ref{10points} to the following set of 10 points.

\begin{table}[h!]
 \centering
 \begin{tabular}{|c|c|c|c|c|}
  \hline
   $p_1$ & $q_1$ &          &          & $t_1$    \\
  \hline
   $p_0$ & $q_0$ & $r_0$    & $s_0$    &     	  \\
  \hline
         &       & $r_{-1}$ & $s_{-1}$ & $t_{-1}$ \\
  \hline
 \end{tabular}
\end{table}

The exact statement that we will prove is slightly different to what we have mentioned,
as we will be proving that a $5$-cell grid at the neighbourhood of a segment of rather
good edges is contained in the intersection of $5$ linearly independent quadrics. We
present the theorem in this way as it will be more useful to us later to state it like
this. However, one can see that the proof we present here can be easily adapted to take 
into account the whole grid.

\begin{theorem} \label{quadricsx}
Let $\mathcal{S}$ be a set of $n$ points in $\mathbb{P}^4({\mathbb R})$ not all in a hyperplane and such
that any $4$ points of $\mathcal{S}$ span a hyperplane.  Let $p, q, r \in \mathcal{S}$ be points of $\mathcal{S}$
 such that the line $p^* \cap q^* \cap r^*$ in $\Gamma$ contains a segment $T$ of $m$
 rather good edges. Then there are $5$ linearly independent quadrics such that they
 contain the points $p$, $q$ and $r$ and all the points $s$ such that $s^*$
 intersects $p^* \cap q^* \cap r^*$ in $T$.
\end{theorem}
\begin{proof}
 First of all, because the segment $T$ is a segment of $m$ rather good edges, by Lemma~\ref{rathergoodedge}, we
 know that the structure of $\Gamma$ around $T$ is almost that of a $5$-cell grid of
 dimensions $3\times3\times3\times (m+1) \times (m+1)$. We clarify this {\em almost} below.
 
Because it has the structure of a $5$-cell grid, we can rename the points
 involved to be as follows:
 
 \begin{equation*}
  \begin{aligned}
			  p_{-1} \quad &p_0 \quad p_1 \\
			  q_{-1} \quad &q_0 \quad q_1 \\
			  r_{-1} \quad &r_0 \quad r_1 \\
   s_{-m} \quad ... \quad s_{-1} \quad &s_0 \\
				       &t_0 \quad t_1 \quad ... \quad t_m
  \end{aligned}
 \end{equation*}
 
 having $p = p_0$, $q = q_0$ and $r = r_0$. Lemma~\ref{rathergoodedge} implies that the hyperplanes $p_i^*$, $q_j^*$, $r_k^*$, $s_l^*$, $t_n^*$ intersect in a point if and only if $i + j + k + l + n = 0$ for all $i,j,k \in \{-1,0,1\}$, $l \in \{-m,\ldots,0\}$ and $n \in \{0,\ldots,m\}$ unless (possibly) $\{i,j,k\} \subseteq \{-1,1\}$.
 
 Now we apply the Lemmas~\ref{11pts} and Lemma~\ref{10points} in order to get
 our $5$ quadrics. It will be important when we use these lemmas that the $11$
 (respectively $10$) points we apply the lemmas to, hold the same structure as
 the points used in the lemmas and that one can form a bijection between the
 sets of points that maintains invariant the spanned hyperplanes of the set. \\
 
 Firstly we apply Lemma \ref{11pts} to the set:
 
\begin{table}[h!]
 \centering
 \begin{tabular}{|c|c|c|c|c|}
  \hline
   $t_1$ & $r_1$ &          &          & $p_1$    \\
  \hline
   $t_0$ & $r_0$ & $s_0$    & $q_0$    & $p_0$	  \\
  \hline
         &       & $s_{-1}$ & $q_{-1}$ & $p_{-1}$ \\
  \hline
 \end{tabular}
\end{table}
%  \begin{equation*}
% %   \{t_0, t_1, r_0, r_1, s_{-1}, s_0, q_{-1}, q_0, p_{-1}, p_0, p_1\}
%   \begin{aligned}
%    t_1 & r_1 &        &        & p_1    \\
%    t_0 & r_0 & s_0    & q_0    & p_0    \\
%        &     & s_{-1} & q_{-1} & p_{-1}
%   \end{aligned}
%  \end{equation*}
 
 It is clear that this set of points holds the same structure as that of the set 
 in Lemma~\ref{11pts}. So Lemma \ref{11pts} applies and we know that there are $5$ linearly
 independent quadrics containing the $11$ points. 
 
 Let us denote by $\{Q_1, Q_2, Q_3, Q_4, Q_5\}$ this set of five quadrics.\\
 
Consider the set of $10$ points.

%\begin{table}[h]
 \begin{center}
 \begin{tabular}{|c|c|c|c|c|}
  \hline
   $t_1$ & $r_1$ &          &          & $q_1$    \\
  \hline
   $t_0$ & $r_0$ & $s_0$    & $p_0$    & 	  \\
  \hline
         &       & $s_{-1}$ & $p_{-1}$ & $q_{-1}$ \\
  \hline
 \end{tabular}
 \end{center}
%\end{table}
%  \begin{equation*}
%   \begin{aligned}
%    t_1 & r_1 &        &        & q_1    \\
%    t_0 & r_0 & s_0    & p_0    & 	\\
%        &     & s_{-1} & p_{-1} & q_{-1}
%   \end{aligned}
%  \end{equation*}

\vspace{.2cm}
 
 This set of points has the same structure as the set of points in Lemma \ref{10points}, so the lemma applies and any quadric passing through $9$ of these points passes
 through the tenth. But  we know that the $9$
 points 
 $$\{t_0, t_1, r_0, r_1, s_{-1}, s_0, p_{-1}, p_0, q_{-1}\}$$ are contained
 in the $5$ quadrics $Q_i$. Thus, the tenth point $q_1$ is also contained in
 these quadrics.
 
 We can use this same trick to argue that the point $r_{-1}$ is also contained
 in the quadrics.\\
 
 Now we want to extend this argument to include all the points $s_{-i}$ and $t_i$.
 We will manage this by induction.
 
 First of all, we already know that the points $\{s_0, s_{-1}, t_0, t_1\}$ are
 contained in the $5$ quadrics. 
 
 Now, for the induction hypothesis, let us suppose that the set of points 
 $\{s_0, s_{-1} ... s_{-i}, t_0, t_1 ... t_i\}$ are all contained in our set 
 of $5$ quadrics. We will now prove that both $s_{-i-1}$ and $t_{i+1}$ are
 also contained in the $5$ quadrics.
 
 To begin with, consider the following set of points:
 
\begin{table}[h!]
 \centering
 \begin{tabular}{|c|c|c|c|c|}
  \hline
   $p_1$ & $q_0$    &          &            & $t_{i+1}$ \\
  \hline
   $p_0$ & $q_{-1}$ & $r_0$    & $s_{-i+1}$ &	      \\
  \hline
         &          & $r_{-1}$ & $s_{-i}$   & $t_{i-1}$ \\
  \hline
 \end{tabular}
\end{table}
%  \begin{equation*}
%   \begin{aligned}
%    p_1 & q_0    &        &        & t_2    \\
%    p_0 & q_{-1} & r_0    & s_0    & 	   \\
%        &     	& r_{-1} & s_{-1} & t_0
%   \end{aligned}
%  \end{equation*}
 
 Although it is slightly more complicated than before, one can check that the same hyperplane relation holds for this set of points as the one in Lemma \ref{10points}. 
Hence, Lemma \ref{10points} applies. Since we know from the induction hypothesis
 that the $9$ points 
 $$\{p_0, p_1, q_{-1}, q_0, r_{-1}, r_0, s_{-i+1}, s_{-i}, t_{i-1}\}$$
 are all contained in the $5$ quadrics $Q_i$, we deduce that the tenth point $t_{i+1}$ 
 is also contained in the quadrics.
 
 On the other hand, let us consider the following points:

\begin{table}[h!]
 \centering
 \begin{tabular}{|c|c|c|c|c|}
  \hline
   $p_{-1}$ & $q_0$ &       &           & $s_{-i-1}$ \\
  \hline
   $p_0$    & $q_1$ & $r_0$ & $t_{i-1}$ & 	     \\
  \hline
            &       & $r_1$ & $t_i$     & $s_{-i+1}$ \\
  \hline
 \end{tabular}
\end{table}

%  \begin{equation*}
%   \begin{aligned}
%    p_{-1} & q_0 &     &     & s_{-2} \\
%    p_0    & q_1 & r_0 & t_0 & 	     \\
%           &     & r_1 & t_1 & s_0
%   \end{aligned}
%  \end{equation*}
 
 As in the previous case we have that the same structure holds. So Lemma \ref{10points} applies, and by the same argument as before, since the
 induction hypothesis tells us that the $9$ points 
 $$\{p_0, p_{-1}, q_0, q_1,
 r_0, r_1, t_{i-1}, t_i, s_{-i+1}\}$$
  are already contained in our $5$ quadrics,
 we get that the tenth point of the set $s_{-i-1}$ is also contained in the 
 quadrics.
\end{proof}

With this theorem we have proven that the dual set of a $5$-cell grid (or a 
segment of rather good edges in $\Gamma$) is contained in the intersection of
$5$ quadrics. Later we will prove using this that any set of points spanning
few ordinary hyperplanes must also be mostly contained in the intersection of
$5$ quadrics.

\section{The variety defined by the intersection of five quadrics}

In this section, we shall study some properties of the variety defined by five linearly independent quadrics, although the results in this section are not necessary in order to be able to prove the main theorem, Theorem~\ref{FullStr}, of the article.

The main theorem of this section is inspired by Glynn's article \cite{Glynn94}. In his
paper, Glynn talks, among other things, about normal rational curves and the 
intersection of quadrics. Although his work is applied to spaces over the finite field 
$\mathbb{F}_q$, most of his results can be adapted to work over an arbitrary field ${\mathbb F}$.

We are particularly interested in the following theorem which is \cite[Theorem 3.1]{Glynn94}. 

We define an {\em arc} of $\mathbb{P}^d({\mathbb F})$ as a set $\mathcal{A}$ of points of
$\mathbb{P}^d({\mathbb F})$ with the property that any subset of $d+1$ points of $\mathcal{A}$ spans $\mathbb{P}^d({\mathbb F})$.

\begin{theorem} \label{Glyn}
 Consider a subspace $Q$ of quadrics on $\mathbb{P}^d({\mathbb F})$ generated by a collection of ${d \choose 2}$ independent quadrics. Let $\mathcal{A} = \bigcap Q$ be the intersection of the quadrics in $Q$. Suppose that $\mathcal{A}$ generates $\mathbb{P}^d({\mathbb F})$ and that $Q$ does not contain any quadric that is the union of two hyperplanes of 
 $\mathbb{P}^d({\mathbb F})$. Then $\mathcal{A}$ is an arc of $\mathbb{P}^d({\mathbb F})$.
\end{theorem}

This theorem is relevant here because it implies that if instead of $5$
linearly independent quadrics we had that our set $\mathcal{S}$ is included in the
intersection of $6$ linearly independent quadrics such that they do not span 
any hyperplane pair quadric, then the set $\mathcal{S}$ would be an arc. By the
definition of arc, this would mean that every hyperplane spanned by the
set $\mathcal{S}$ would be ordinary, which is precisely the contrary of our aim here.

However, Glynn's theorem can be extended to the following theorem.

\begin{theorem} \label{glynn+}
Consider a subspace $Q$ of quadrics on $\mathbb{P}^d({\mathbb F})$ generated by a collection of ${d \choose 2}-1$ independent quadrics. Let $\mathcal{S} = \bigcap Q$ be the intersection of the quadrics in $Q$. Suppose that $\mathcal{S}$ generates $\mathbb{P}^d({\mathbb F})$, does not contain a line, and that $Q$ does not contain any quadric that is the union of two hyperplanes of 
 $\mathbb{P}^d({\mathbb F})$.  Then $\mathcal{S}$ has the property that any hyperplane contains at most $d+1$ points of $\mathcal{S}$.
\end{theorem}

\begin{proof}
 Suppose that this is false, and let $H$ be a hyperplane containing at least $d+2$
 points of $S$. 
 
Since a vector space satisfies the exchange axiom and $\mathcal{S}$ spans the whole space, there is a hyperplane $H$ which contains at least $d+2$ points $\mathcal{S}$, $d$ of which span $H$. After a suitable change of basis, we can assume that $H$ is the hyperplane $X_{d+1} = 0$ and that the $d+2$ points are 
 
 \begin{equation*}
  \begin{aligned}
   p_1 &= (1:0: 0: \ldots :0:0) \\
   p_2 &= (0: 1: 0: \ldots: 0:0) \\
   \vdots  \\
   p_{d-1} &= (0: \ldots:0:1: 0: 0) \\
   p_d &= (0: 0: \ldots:0: 1: 0) \\
   p_{d+1} &= (c_1: c_2: \ldots:c_d: 0) \\
   p_{d+2} &= (e_1: e_2: \ldots:e_d: 0)
  \end{aligned}
 \end{equation*}
 
 We want to prove that these $d+2$ points impose distinct conditions on the space of quadrics containing them. Clearly $p_1, \ldots, p_d$ impose distinct conditions. It is easy to check that the two last points impose two further distinct conditions on the
 space of quadrics. Indeed, the points $p_{d+1}$ and $p_{d+2}$ will impose linearly independent
 conditions in the space of quadrics unless the $2 \times 2$ sub-determinants of the matrix
 
$$
\left(\begin{array}{cccccc}
   c_1c_2 & c_1c_3 & c_1c_4 &  \ldots & c_{d-2}c_d & c_{d-1}c_d \\
   e_1e_2 & e_1e_3 & e_1e_4 & \ldots & e_{d-2}e_d & e_{d-1}e_d
  \end{array} \right)
$$
 
 are zero. This can only happen if the two points are the same or if only one of the 
 coefficients $X_iX_j$ is non-zero, which means that the points $p_{d+1}$ and $p_{d+2}$ are
 on a line joining two of the basis points. But then this would imply that the quadrics, and hence $\mathcal{S}$, would contain this line which by hypothesis does not happen. Hence, $p_{d+1}$ and $p_{d+2}$ impose different conditions on the space of quadrics.

 So the $d+2$ points $p_1,\ldots,p_{d+2}$ impose $d+2$ linearly independent conditions on the space of
 quadrics defined on the hyperplane $H \cong \mathbb{P}^{d-1}({\mathbb F})$. This space of quadrics has dimension ${d+1 \choose 2}$ which implies that we have at most ${d+1 \choose 2}-(d+2)=\frac{1}{2}d^2-\frac{1}{2}d-2$ linearly 
 independent quadrics on $H$ containing the points $\mathcal{S} \cap H$.
 
 As we have ${d \choose 2}-1=\frac{1}{2}d^2-\frac{1}{2}d-1$ linearly independent quadrics which are zero on $\mathcal{S}$, this means that
 in the subspace of quadrics generated by them, there must be two independent
 quadrics that agree on their restriction to $H$. Taking the difference of these two quadrics, this implies that there is a quadric
 in the subspace generated by the ${d \choose 2}-1$ quadrics which is zero on $H$. This quadric is necessarily a hyperplane pair quadric, which contradicts the hypothesis.
\end{proof}

In our case when $d=4$, Theorem~\ref{glynn+} implies that if $\mathcal{S}= \bigcap Q$ spans the whole space and does not contain a line then it has the property that any hyperplane of $\mathbb{P}^4({\mathbb R})$ intersects $\mathcal{S}$ in at most $5$ points.

\section{Structural theorems in $4$ dimensions}

%After all the work done in the previous sections, here we are ready 
In this section we prove structure theorems for four dimensional space. Firstly, we present a weak version of the structure theorem. This version of the structure theorem, although a weaker result than the main structure theorem we will
present later, has merit of its own, since we do not make any assumptions with respect to the value of $K$, in contrast with the main structure theorem, where we need $K$ to be $o(n^{\frac{1}{7}})$.

We will state our weak structure theorem in the same fashion as
Green and Tao. The exact numbers in the theorem are not that important, but it is important that we can prove that the set $\mathcal{S}$ is contained in the union of an $O(K)$ number of varieties, each of which is the intersection of $5$ quadrics.

\begin{theorem}[Weak structure theorem in $4$ dimensions]
 Let $\mathcal{S}$ be a set of $n$ points in $\mathbb{P}^4({\mathbb R})$, such that any 
 subset of $4$ points span a hyperplane and $\mathcal{S}$ is not contained 
 in a hyperplane. If $K\geqslant 1$ and $\mathcal{S}$ spans at most $Kn^3$ ordinary hyperplanes, 
 then $\mathcal{S}$ is contained in the union of at most $33699K$ varieties, 
 each of which is the intersection of $5$ quadrics.
\end{theorem}

\begin{proof}
 From Lemma \ref{sbedges} we get that there are at most $1872Kn^3$ slightly bad edges
 in the graph $\Gamma$. Using the pigeon-hole principle we know that there have to
 be points $p, q, r \in S$ such that the number of slightly bad edges on the line
 $l = p^* \cap q^* \cap r^*$ is at most
 
 \begin{equation*}
  \frac{1872Kn^3}{{n \choose 3}} = 11232K\frac{n^2}{(n-1)(n-2)}=11232K\left(1+\frac{3n-2}{(n-1)(n-2)}\right).
 \end{equation*}
 
 We can assume that $n \geqslant 33699$ since otherwise the theorem is immediately proven
 by choosing a variety for each point. This alone is enough to bound the number of
 slightly bad edges $b$ of $l$ by $11233K$.
 
 The slightly bad edges will partition the line $l$ into a set of at most $b$ 
 segments of consecutive good edges. By Theorem \ref{quadricsx}, we know that for any
 of these segments $T$, there are $5$ linearly independent quadrics containing the
 points $p, q, r$ and such that for any point $s \in \mathcal{S}$, where $s^*$ intersects the
 line $l$ in $T$, is contained in the intersection of the $5$ quadrics. So we can
 cover all the points whose duals intersects $l$ in the middle of a segment of good
 edges with less than $11233K$ of these varieties.
 
 On the other hand, any point whose dual intersects $l$ in a vertex of a bad edge
 of $l$ has to be treated separately. We have, though, at most $2 \times 11233K$
 vertices incident with a bad edge in $l$, and for each one of these vertices, we 
 can construct a variety (intersection of $5$ quadrics) such that all the points 
 whose dual intersects $l$ in that vertex are contained in the variety (this is obvious,
 since the dual of the vertex itself is a hyperplane that contains all of these
 points).
 
 From the number of varieties covering the points incident with the segments of good
 edges and the ones incident with the bad edges, we conclude that all the points of 
 $ \mathcal{S}$ can be covered with a collection of at most $11233K + 2\times11233K = 33699K$ 
 varieties each of which is the intersection of $5$ quadrics.
\end{proof}

Now we present the full structure theorem, which is the main theorem of
this section, as well as the whole paper. The purpose of the theorem is to classify
the sets of points in $\mathbb{P}^4({\mathbb R})$ that span few ordinary hyperplanes, as did the structure theorems of Green and Tao \cite{GT2013} and Ball \cite{Ball2018} in two and three dimensions respectively. 

\begin{theorem} [Full Structure Theorem] \label{FullStr}
 Let $\mathcal S$ be a set of $n$ points in $\mathbb{P}^4({\mathbb R})$ such that any subset of $4$ points
 spans a hyperplane and such that $\mathcal S$ is not contained in a hyperplane. If $\mathcal S$
 spans less than $Kn^3$ ordinary hyperplanes, for some $K = o(n^{\frac{1}{7}})$, then
 one of the following holds:
 \begin{enumerate}[(i)]
  \item\label{strucase1} All but at most $6K$ points of $\mathcal S$ are contained in a hyperplane.
  \item\label{strucase2} There are $5$ linearly independent quadrics such that all but 
  at most $O(K)$ points of $\mathcal S$ are contained in the intersection of the $5$ quadrics. 
 \end{enumerate}

\end{theorem}
\begin{proof}
% We will make the proof of the statement in several steps. \\
 
%Firstly we consider the projection $\mathcal{S}$ to $3$ and $2$ dimensions, and make use of the respective structure theorems in those dimensions. Looking at the possible
% structures of the projection of our set will help us to classify and discard
% the different possibilities that occur.
 
%Then we will need to extract from the structure of the projections the necessary
% information to argue the existence of the $5$ linearly independent quadrics
% that contain the set.\\
 
 Let $\mathcal{S}'$ be the subset of points $p \in \mathcal{S}$ such that $p$
 is incident with at most $DKn^2$ ordinary hyperplanes, for some large constant
 $D$. We can bound the size of $\mathcal{S}'$ in the following way:
 
Since the points of $\mathcal{S}$ not in $\mathcal{S}'$ are incident with more than $DKn^2$ ordinary
 hyperplanes, and since every ordinary hyperplane contains exactly $4$ points,
 we get
 
 \begin{equation*}
  |S \backslash S'| DKn^2 < 4Kn^3.
 \end{equation*}
 
This gives
 
 \begin{equation} \label{Spbound}
  |S'| > \left(1 - \frac{4}{D} \right)n.
 \end{equation}

 If we project the set $\mathcal{S}$ from any point $p \in \mathcal{S}'$ we obtain a set in $\mathbb{P}^3({\mathbb R})$ spanning less than $DKn^2$ ordinary planes.

 The three-dimensional structure theorem from \cite[Theorem 1]{Ball2018} tells us the different possible sets that span few ordinary planes in $\mathbb{P}^3({\mathbb R}
)$. These possibilities are:
 
 \begin{enumerate}[(a)]
  \item\label{stru3cas1} There are two distinct quadrics such that all but at most 
  $O(K)$ points of $\mathcal{S}$ are contained in the intersection of the quadrics. %And all but 
  %at most $O(K)$ points of $\mathcal{S}$ are incident with at least $\frac{3}{2}n - O(K)$ ordinary
 % planes. 
  \item\label{stru3cas2} There are two planar sections of a quadric which contain 
  $\frac{1}{2}n - O(K)$ points of $\mathcal{S}$ each. 
  \item\label{stru3cas3} All but at most $2DK$ points of $\mathcal{S}$ are contained in a plane. 
 \end{enumerate}
 
 First assume that there is a point in $\mathcal{S}'$ where the projection of the set $\mathcal{S}$ from $p$ is
 almost contained in a plane (case \eqref{stru3cas3}).
 
 This implies that there is a solid $\pi$ containing at least $n-c$ points of $\mathcal{S}$, where $c=O(K)$. Let $q$ be a point of $\mathcal{S} \setminus \pi$. Each triple of points $\mathcal{S} \cap \pi$, together with $q$ span an ordinary solid, unless there is another point $p\in \mathcal{S} \setminus (\pi \cup \{q\})$, on the solid. However, for each pair of points $x,y$ of $\mathcal{S} \cap \pi$, consider the plane of $\pi$ spanned by $x,y$ and the point of intersection of line joining $p$ and $q$ with $\pi$. By hypothesis, this plane contains at most one other point of $\mathcal{S}$. Thus, we conclude that, for a fixed $x,y\in\mathcal{S} \cap \pi$, all but at most $c-1$ of the triples $x,y,z$, where $z\in\mathcal{S} \cap \pi$, span an ordinary solid with $q$. Therefore, there are at least 
 $$
 {n -c\choose 3}-(c-1){n-c \choose 2}
 $$ 
 triples of $\pi \cap \mathcal{S}$ which together with $q$ span an ordinary solid.
 
 Hence, 
 $$
 c\Big({n -c\choose 3}-(c-1){n-c \choose 2}\Big) \leqslant Kn^3.
 $$ 
 For $n$ large enough, this implies that $c \leqslant 6K$ and we have case (i).
 
We can suppose from now on that no point of $\mathcal{S}'$ projects the set $\mathcal{S}$ into
 a set almost contained in a plane.\\
 
We want to prove that almost none of the points of $\mathcal{S}'$ can project the set 
 $\mathcal{S}$ into a set of type \eqref{stru3cas2}.
 
 Let us suppose that there are at least four points in $\mathcal{S}'$ such that the projection
 of $\mathcal{S}$ from these points is as in case \eqref{stru3cas2}. We denote these points by $p_1, p_2, p_3, p_4$. This means that, for each of these points, there are two 
 planar conics containing $\frac{n}{2} - O(K)$ points of the projection each.
 
 The lift of these two planes to $\mathbb{P}^4(\mathbb{R})$ consists of two hyperplanes $H_1$
 and $H_2$ containing $\frac{n}{2} - O(K)$ points of $\mathcal{S}$ each.
 
 We claim that the lift of the planes produced by each of the points $p_i$ must result
 in the same two hyperplanes $H_1$ and $H_2$. Suppose not, that the lift
 of one of the planes produced by some of the $p_i$ produces a hyperplane $H_3$,
 different from $H_1$ and $H_2$. The intersections $\pi_1 = H_1 \cap H_3$ and 
 $\pi_2 = H_2 \cap H_3$ are both planes on $\mathbb{P}^4(\mathbb{R})$. Since all $H_1, H_2$
 and $H_3$ contain $\frac{n}{2} - O(K)$ points of $\mathcal{S}$, one of
 $\pi_1$ or $\pi_2$ would have at least $\frac{n}{4} - O(K)$ points of $\mathcal{S}$,
 which is a contradiction since we cannot have four coplanar points in $\mathcal{S}$.
 
 Now, as all $p_i$ project the points in the two hyperplanes $H_1$ and $H_2$ into
 two planar conics, we deduce that the points $p_i$ are all contained in the
 hyperplanes $H_1$ and $H_2$. That implies they are in the intersection of the 
 two hyperplanes which, since the hyperplanes are distinct, is a plane of
 $\mathbb{P}^4(\mathbb{R})$. By hypothesis, the set $\mathcal{S}$ does not have four coplanar points, a contradiction.
 
 Thus, we conclude that there are at most three points of $\mathcal{S}'$ which project the
 set $\mathcal{S}$ into a set of type \eqref{stru3cas2}.\\
 
 Therefore, removing at most three points from $\mathcal{S}'$, we can assume that all points of $\mathcal{S}'$ project the
 set $\mathcal{S}$ into a set of type \eqref{stru3cas1}.

 Now, let us consider two of these points $p_1$ and $p_2 \in \mathcal S'$, projecting $\mathcal S$ onto a set of type (a). Apply a
 projective transformation so they become the points $(1: 0 : 0 : 0 : 0)$ and
$(0: 1 : 0 : 0 : 0)$ respectively.
 
 Let $Q_1$ and $Q_1'$ be the two quadrics degenerate at $p_1$ given by (a), and let $Q_2$ and $Q_2'$ be the quadrics degenerate at $p_2$, also given by (a). Then these
 quadrics will have coefficient $0$ at the $X_1$ and $X_2$ terms respectively.
 
 %\begin{equation*}
  %\begin{aligned}
  % Q_1  &= a_{22}X_2^2 + \dots + a_{55}X_5^2 + a_{23}X_2X_3 + \dots a_{45}X_4X_5 \\
  % Q_1' &= b_{22}X_2^2 + \dots + b_{55}X_5^2 + b_{23}X_2X_3 + \dots b_{45}X_4X_5 \\
  % Q_2  &=  c_{22}X_1^2 + c_{33}X_3^2 + \dots + c_{55}X_5^2 + c_{13}X_1X_3 + \dots c_{45}X_4X_5 \\
  % Q_2' &= d_{22}X_1^2 + d_{33}X_3^2 + \dots + d_{55}X_5^2 + d_{13}X_1X_3 + \dots d_{45}X_4X_5
  %\end{aligned}
 %\end{equation*}

First we prove that if $Q_1, Q_1'$ and $ Q_2, Q_2'$ are not linearly independent quadrics then $p_1$ and $p_2$ are contained in many ordinary solids. Indeed, if $Q_1, Q_1'$ and $ Q_2, Q_2'$ are linearly dependent quadrics then there are constants $\lambda, \lambda', \mu, \mu'$, not all zero, such that
$$
\lambda Q_1+\lambda' Q_1'+\mu Q_2+\mu' Q_2'=0.
$$
Therefore,  the quadric 
$$
Q=\lambda Q_1+\lambda' Q_1'=-\mu Q_2-\mu' Q_2'=e_{34}X_3X_4+e_{35}X_3X_4+e_{45}X_4X_5,
$$
is degenerate at both $p_1$ and $p_2$.

The quadric $Q$ contains all but at most $O(K)$ points of $\mathcal{S}$, so the projection of $\mathcal{S}$ from $p_1$ and $p_2$, projects all but at most $O(K)$ points of $\mathcal{S}$ onto a (possibly degenerate) conic. By hypothesis, any four points of $\mathcal S$ span a solid, so no two points project to the same point of the conic. Suppose $x \in \mathcal S$ projects onto the conic. Together with any one of all but at most $O(K)$ of the other points of $\mathcal S$ the point $x$ spans an ordinary solid with $p_1$ and $p_2$. This gives at least $\frac{1}{2}n^2-O(Kn)$ ordinary solids, containing $p_1$ and $p_2$. Hence, there are less than $3Kn$ pairs $p_1,p_2$ for which $Q_1, Q_1'$ and $ Q_2, Q_2'$ are not linearly independent quadrics.
 
A simple count shows that we can choose the points $p_1,p_2,p_3 \in \mathcal S'$ so that they are all of type (a) and are all contained in the intersection of $ Q_1, Q_1',Q_2, Q_2',Q_3,Q_3'$ and that $Q_1, Q_1'$ and $ Q_2, Q_2'$ are linearly independent quadrics. Indeed, the number of triples of points of $\mathcal S'$ which we cannot choose is $O(Kn^2)$.
Here $Q_i, Q_i'$ are the quadrics which are degenerate at $p_i$ given by $(a)$. 

Then we have that
  \begin{equation*}
  \begin{aligned}
   Q_3(X)  &= e_{12}X_1X_2 + e_{14}X_1X_4+e_{15}X_1X_5+  e_{24}X_2X_4+e_{25}X_2X_5+e_{44}X_4^2+ e_{55}X_5^2 +  e_{45}X_4X_5 \\
   Q_3'(X) &= f_{12}X_1X_2 + f_{14}X_1X_4+f_{15}X_1X_5+  f_{24}X_2X_4+f_{25}X_2X_5+f_{44}X_4^2+ f_{55}X_5^2 +  f_{45}X_4X_5\\
  \end{aligned}
 \end{equation*}
 Let $b_3(X,Y)$ and $b_3'(X,Y)$ be the bilinear forms associated with the quadratic forms $Q_3$ and $Q_3'$ respectively. If $x$ is a common zero of both $Q_3$ and $Q_3'$ then all the points on the line joining $x$ and $p_1$ are zeros of
 $$
 b_3(X,p_1)Q_3'(X)-b_3'(X,p_1)Q_3(X).
 $$
 One can think of this as the cubic curve in the plane which is obtained by projecting the quadrics $Q_3$ and $Q_3'$ from $p_1$ and $p_3$. Calculating this curve explicitly, we obtain
 $$
 ( e_{12}X_2 + e_{14}X_4+e_{15}X_5)( f_{24}X_2X_4+f_{25}X_2X_5+f_{44}X_4^2+ f_{55}X_5^2 +  f_{45}X_4X_5)-
 $$
 $$
 (f_{12}X_2 + f_{14}X_4+f_{15}X_5)( e_{24}X_2X_4+e_{25}X_2X_5+e_{44}X_4^2+ e_{55}X_5^2 +  e_{45}X_4X_5).
 $$
If both the two quadrics $Q_3, Q_3'$ have zero coefficient at 
 $X_1X_2$ (i.e. $e_{12}=f_{12}=0$) then we see that the point $p_2$ is a singularity of the cubic curve (since the degree of the polynomial in $X_2$ is one). 
 
 Since there is a most one singularity, we can find three points $p_1,p_2,p_3 \in {\mathcal S}'$ which are all of type (a), have the property that $p_i \in Q_j \cap Q_j'$, and for which $p_2$ 
is not a singularity of the cubic curve obtained by projecting the quadrics $Q_3$ and $Q_3'$ from the points $p_1$ and $p_3$. Moreover, $Q_1, Q_1', Q_2, Q_2'$ are linearly independent quadrics. 

As we have just seen, in this case at least one of the quadrics $Q_3$ or $Q_3'$ is linearly independent of $Q_1, Q_1',Q_2, Q_2'$. Hence, all but at most $O(K)$ points of $\mathcal{S}$ are contained in the intersection of five  linearly independent quadrics.
\end{proof}

\section{Comments and conjectures}

The most natural question to ask is if one can strengthen Theorem~\ref{FullStr} (ii) to state that all but $O(K)$ points of $\mathcal{S}$ are contained in the lift of a coset of a subgroup of an elliptic curve. This may involve a strengthening of the hypothesis on $K$. 

It seems reasonable to expect that one can improve the bound $K=o(n^{1/7})$ to $K=o(n^{1/6})$ by relying only on \cite[Proposition 5.3]{GT2013} and bypassing \cite[Theorem 1]{Ball2018}. However, we were not able to do this.

%Observe that for each point $p$ in the intersection of the quadrics, there are two quadrics in the subspace of quadrics which are degenerate at $p$ This implies that the projection of the intersection of the quadrics to the plane from $p$ and any other point in the intersection of the quadrics, is a cubic curve.

One can also consider the problem in higher dimensions, and the work included here gives us some clues to its solution.

It is easy to see that some of the concepts we have defined throughout, such as the $5$-cell grid, or the concepts of good and bad edges, are
naturally generalised to higher dimensions, the higher dimensional problem being first proposed in \cite{BM2016}. But the key ingredient of our study,
the nature of the variety containing almost all the points of the sets spanning
few ordinary hyperplanes, is a bit harder to generalise to higher dimensions.

We have talked briefly of the relation between the key varieties of the examples
in two, three and four dimensions. 
The relation between the varieties of different dimensions is very natural. It
comes from the fact that, if $\mathcal{S}$ is a set with few ordinary hyperplanes, the
projection of $\mathcal{S}$ from most of its points is a set with few ordinary hyperplanes
in a lower dimension. Thus, it is only natural that the varieties containing
these set of points are lifts of the other.

The work of Glynn \cite{Glynn94} gives us reasonable enough evidence to conjecture the nature
of these varieties.

\begin{conjecture} \label{conj12}
 Let $\mathcal{S}$ be a set of $n$ points on $\mathbb{P}^d({\mathbb R})$, where $d \geqslant 4$ is a constant or a possibly sub-linear function of $n$, such that any subset of $d$ points of $\mathcal{S}$ span a hyperplane of $\mathbb{P}^d({\mathbb R})$. Suppose that $\mathcal{S}$ spans at most
 $Kn^{d-1}$ hyperplanes for some $K=o(n)$, then one of the following holds:
 \begin{enumerate}
  \item[(i)]
   All but at most $O(K)$ points of $\mathcal{S}$ are contained in a single hyperplane.
  \item[(ii)] There is a set of ${d \choose 2} - 1$ linearly independent quadrics
  such that all but at most $O(K)$ points of $\mathcal{S}$ are contained in the intersection
  of the quadrics.
 \end{enumerate}
\end{conjecture}

We could go further and conjecture that all but at most $O(K)$ points of $\mathcal{S}$ are contained in the lift of a coset of a subgroup of an elliptic curve. This may require a strengthening of the hypothesis on $K$.

Since the publication of a preprint of this article, the following theorem has appeared in \cite[Theorem 1.1]{LS2020}, which proves this stronger version of Conjecture~\ref{conj12}, albeit with a stronger hypothesis.

\begin{theorem}
 Let $\mathcal{S}$ be a set of $n$ points of $\mathbb{P}^d({\mathbb R})$, where $d \geqslant 4$ and $C\geqslant \mathrm{max} ((dK)^8,d^32^dK)$ for some sufficiently large
absolute constant $C$, such that any subset of $d$ points of $\mathcal{S}$ span a hyperplane of $\mathbb{P}^d({\mathbb R})$. If $\mathcal{S}$ spans at most
 $K{n-1 \choose d-1}$ hyperplanes then $\mathcal S$ differs by at most $O(d2^dK)$ points from one of the following:
 \begin{enumerate}
  \item[(i)]
A subset of a hyperplane.
  \item[(ii)] A coset of a subgroup of an elliptic curve or the smooth points of a rational acnodal curve of degree $d+1$.
 \end{enumerate}
\end{theorem}

\section{Acknowledgements}

The authors would like to thank Massimo Giulietti for some useful discussions about lifts of elliptic curves. We also thank the referee whose detailed comments and suggestions were greatly appreciated.

\bigskip

{\small Simeon Ball and Enrique Jimenez}  \\
{\small Departament de Matem\`atiques}, \\
{\small Universitat Polit\`ecnica de Catalunya, Jordi Girona 1-3},
{\small M\`odul C3, Campus Nord,}\\
{\small 08034 Barcelona, Spain} \\
{\small {\tt simeon@ma4.upc.edu}}


\begin{thebibliography}{}

\bibitem{Ball2018} S. Ball, On sets defining few ordinary planes, {\it Discrete Comput. Geom.}, {\bf  60} (2018) 220--253.

\bibitem{BM2016} S. Ball and J. Monserrat, A generalisation of Sylvester's problem to higher dimensions, {\em J. Geom.}, {\bf 108} (2017) 529--543.

\bibitem{BGP2015} D. Bartoli, M. Giulietti and I. Platoni, On the covering radius of MDS codes, {\em IEEE Trans. Inform. Theory}, {\bf 61} (2015) 801--811.

\bibitem{BVZ2016} T. Boys, C. Valculescu and F. de Zeeuw, On the number of ordinary conics, {\em SIAM J. Discrete Math.}, {\bf 30} (2016) 1644--1659. 

%\bibitem{Bacharach1886} I. Bacharach, Ueber den Cayley'schen Schnittpunktsatz, {\it Mathematische Annalen}, {\bf 26} (1886) 275--299.

%\bibitem{BM1990} P. Borwein, W. O. J. Moser, A survey of Sylvester's problem and its generalizations, {\it Aequationes Mathematicae}, {\bf 40} (1990) 111--135.

%\bibitem{Cayley1889} A. Cayley, {\it On the Intersection of Curves}, Cambridge University Press, Cambridge, 1889.

%\bibitem{Chasles1885}, M. Chasles, {\it Traité des sections coniques}, Gauthier-Villars, Paris, 1885.

%\bibitem{CS1993} J. Csima and E. Sawyer, There exist $6n/13$ ordinary points, {\it Discrete and Computational Geometry}, {\bf 9} (1993) 187--202.

\bibitem{CDFGLMSST2016} A. Czapli\'nski, M. Dumnicki, \L . Farnik, J. Gwo\' zdziewicz, M. Lampa-Baczy\' nska, G. Malara, T. Szemberg, J. Szpond and H. Tutaj-Gasi\'nska, On the Sylvester-Gallai theorem for conics. {\it Rend. Semin. Mat. Univ. Padova}, {\bf 136} (2016) 191--203.

%\bibitem{Donagi1980} R. Donagi, Group law on the intersection of two quadrics, {\it Annali della Scuola Normale Superiore de Pisa}, {\bf 7} (1980) 217--239.

%\bibitem{Gallai1944} T. Gallai, Solution to Problem 4065, {\it American Mathematical Monthly}, {\bf 51} (1944) 169--171.

\bibitem{Glynn94} D. G. Glynn, On the construction of arcs using quadrics, {\it Australas. J.Combin.}, {\bf 9} (1994) 3--19.

\bibitem{GT2013} B. Green and T. Tao, On sets defining few ordinary lines, {\it Discrete Comput. Geom.}, {\bf  50} (2013) 409--468.

\bibitem{LS2020} A. Lin and K. Swanepoel, On sets defining few ordinary hyperplanes, {\it Discrete Anal.},  {\bf 4} (2020) 34pp.

\bibitem{LMMSSdZ2016} A. Lin, M. Makhul, H. N. Mojarrad, J. Schicho, K. Swanepoel, F. de Zeeuw, On sets defining few ordinary circles, {\it Discrete Comput. Geom.},  {\bf 59} (2018) 59--87.

%\bibitem{Hesse1840} O. Hesse, De curvis et superficiebus secundi ordinis, {\em Journal f\" ur die reine und angewandte Mathematik}, {\bf 20} (1840) 285--308.  

%\bibitem{KM1958} L. Kelly, W. Moser, On the number of ordinary lines determined by $n$ points, {\it Canadian Journal of Mathematics}, {\bf 10} (1958) 210--219.

%\bibitem{LMMSSdZ2016} A. Lin, M. Makhul, H. N. Mojarrad, J. Schicho, K. Swanepoel, F. de Zeeuw, On sets defining few ordinary circles, {\it Discrete and
%Computational Geometry}, to appear. {\tt doi:10.1007/s00454-017-9885-8.}

\bibitem{Melchior1940} E. Melchior, \"Uber Vielseite der projektiven Ebene, {\it Deutsche Mathematik}, {\bf 5} (1941) 461--475.

%\bibitem{Monserrat2015} J. Monserrat, {\it Generalization of Sylvester problem}, Bachelors Degree Thesis, Universitat Polit\`ecnica Catalunya, Barcelona, 2015.

%\bibitem{Motzkin1951} T. Motzkin, The lines and planes connecting the points of a finite set, {\it Transactions of the American Mathematical Society}, {\bf 70} (1951) 451--464.

%\bibitem{Pedoe1988} D. Pedoe, {\it Geometry, A Comprehensive Course}, Dover, 1988. (Corrected republication of {\it A Course in Geometry for Colleges and Universities}, Cambridge, 1970.)

%\bibitem{Syl1868} J. Sylvester, Mathematical Question 2571. Educational Times, February 1868.

\bibitem{Silverman2009} J. H. Silverman, {\em The Arithmetic of Elliptic Curves}, Springer, 2009.

\bibitem{Syl1893} J. Sylvester, Mathematical Question 11851, Educational Times, {\bf 59} (1893) 98.

\end{thebibliography}
\end{document}